\documentclass[12pt]{amsart}
\usepackage{amssymb}
\usepackage{amsmath}
\usepackage{amsthm}
\usepackage[all,arc,curve,color,frame]{xy}

\DeclareMathOperator{\ad}{ad}

\newcommand{\A}{{\mathcal A}}
\newcommand{\E}{{\mathcal E}}
\newcommand{\F}{{\mathcal F}}
\newcommand{\I}{{\mathcal I}}
\renewcommand{\i}{\mathbf {in}}
\renewcommand{\o}{\mathbf {out}}
\renewcommand{\L}{{\mathcal L}}
\renewcommand{\d}{{\mathbf d}}

\newcommand{\n}{{\mathbf n}}
\newcommand{\m}{{\mathbf m}}
\newcommand{\pp}{{\mathbf p}}
\newcommand{\qq}{{\mathbf q}}

\newcommand{\N}{{\mathbb N}}
\newcommand{\C}{\ensuremath{\mathbb{C}}}
\newcommand{\R}{\ensuremath{\mathbb{R}}}

\newcommand{\p}{\partial}

\newcommand{\D}{{\mathcal D}}
\newcommand{\unit}{{\bf{1}}}

\newcommand{\U}{{\mathcal U}}
\newcommand{\V}{{\mathcal V}}

\newcommand{\Z}{\ensuremath{\mathbb{Z}}}
\newtheorem{lemma}{Lemma}
\newtheorem{proposition}{Proposition}
\newtheorem{theorem}{Theorem}

\begin{document}

\title[Gammelgaard's formula for a star product]
{On Gammelgaard's formula for a star product with separation of
variables}
\author[Alexander Karabegov]{Alexander Karabegov}
\address[Alexander Karabegov]{Department of Mathematics, Abilene
Christian University, ACU Box 28012, Abilene, TX 79699-8012}
\email{axk02d@acu.edu}

\dedicatory{This paper is dedicated to the memory of Boris Vasil'evich Fedosov}

\begin{abstract} 
We show that Gammelgaard's formula expressing a star product with
separation of variables on a pseudo-K{\"a}hler manifold in terms of directed graphs without cycles is equivalent to an inversion formula for an operator on a formal Fock space. We prove this inversion formula directly and thus offer an alternative approach to Gammelgaard's formula which gives more insight into the question why the directed graphs in his formula have no cycles.
\end{abstract}
\date{April 7, 2012}
\maketitle

\section{Introduction}

Given a vector space $V$, we denote by $V[\nu^{-1},\nu]]$ the set of formal Laurent series in a formal parameter $\nu$ with a finite polar part and coefficients in $V$.
We call the elements of  $V[\nu^{-1},\nu]]$,
\[
      v = \sum_{r \geq n} \nu^r v_r,
\]
where $n \in \Z$ and $v_r \in V$, formal vectors. 

Deformation quantization on a Poisson manifold $(M,\{\cdot,\cdot\})$ is the structure of an associative algebra on the space $C^\infty(M)[\nu^{-1},\nu]]$ of formal functions given by the formula
\begin{equation}\label{E:star}
  f \ast g = \sum_{r=0}^\infty \nu^r C_r (f,g),
\end{equation}
where $C_r$ are bidifferential operators on $M$ such that $C_0(f,g) =fg$ and
\begin{equation}\label{E:pois}
C_1(f,g) - C_1(g,f)= i\{f,g\}. 
\end{equation}
The product $\ast$ is called a star product. It is also assumed that the unit constant function $\unit$ is the unity of the star product, $f \ast \unit = \unit \ast f = f$. This condition means that the operators $C_r$ with $r>0$ annihilate constants, $C_r(f, \unit)=0$ and $C_r(\unit,f)=0$.
 A star product can be restricted (localized) to any open subset of $M$. In what follows we denote by $L_f$ the operator of left star multiplication by $f$ and by $R_g$ the operator of right star multiplication by $g$ so that $L_f g = f \ast g = R_g f$. The associativity of a star product is equivalent to the statement that $[L_f,R_g]=0$ for any $f,g$.

The problem of existence and classification of star products on arbitrary Poisson manifolds was formulated in \cite{BFFLS} and settled in \cite{K}. In \cite{K}  Kontsevich obtained an explicit formula of a star product on $\R^n$ endowed with an arbitrary Poisson structure. This star product on $\R^n$ was expressed in terms of directed graphs.  On symplectic manifolds, where the Poisson structure is nondegenerate, a nice global geometric construction of star products was given by Fedosov in \cite{F}.

If $M$ is a complex manifold, a star product (\ref{E:star}) on $M$ is called a star product with separation of variables if the operators $C_r$ differentiate their first argument in antiholomorphic directions and the second argument in holomorphic ones. Formula (\ref{E:pois})   implies then  that  local holomorphic functions $a,a'$ and local antiholomorphic functions $b,b'$ Poisson commute, $\{a,a'\}=0$ and $\{b,b'\}=0$. Thus the Poisson structure is given by a Poisson tensor of type (1,1) with respect to the complex structure which can be written as $g^{lk}$ in local holomorphic coordinates $\{z^k, \bar z^l\}$.  We call such tensor $g^{lk}$ a K\"ahler-Poisson tensor and the manifold  $(M,\{\cdot,\cdot\})$ a K\"ahler-Poisson manifold.  We have
\begin{equation}\label{E:explpois}
     \{f,g\} = i g^{lk}\left(\frac{\p f}{\p z^k}\frac{\p g}{\p \bar z^l} - \frac{\p g}{\p z^k}\frac{\p f}{\p \bar z^l}\right),
\end{equation}
where the Einstein summation convention over repeated indices is used. It follows from (\ref{E:pois}), (\ref{E:explpois}), and the condition of separation of variables that 
\[
C_1(f,g) = g^{lk} \frac{\p f}{\p \bar z^l}\frac{\p g}{\p z^k}.
\]
The condition that $\ast$ is a star product with separation of variables can be restated as follows: for locally defined holomorphic function $a$ and antiholomorphic function $b$ the operators $L_a=a$ and $R_b =b$ are pointwise multiplication operators, i.e., $a \ast f = af$ and $f \ast b = bf$. It is not yet known whether star products with separation of variables exist on arbitrary K\"ahler-Poisson manifolds. If a K\"ahler-Poisson tensor $g^{lk}$ is nondegenerate, its inverse $g_{kl}$ is a pseudo-K\"ahler metric tensor which defines a global pseudo-K\"ahler form
\[
                      \omega_{-1} = i g_{kl} dz^k \land d\bar z^l
\]
on $M$. The existence of star products with separation of variables on arbitrary pseudo-K\"ahler manifolds was established independently in \cite{CMP1} and \cite{BW}. All star products with separation of variables on a  pseudo-K\"ahler manifold $(M, \omega_{-1})$  were classified in \cite{CMP1}. It was shown that such star products can be bijectively parameterized by the formal closed $(1,1)$-forms   $\omega = (1/\nu)\omega_{-1} + \omega_0 + \nu \omega_1 + \ldots$.

Star products with separation of variables originate in the context of Berezin's quantization of K\"ahler manifolds (see \cite{Ber}). Examples of star products with separation of variables obtained from Berezin's quantization on homogeneous K\"ahler manifolds were considered in \cite{Mor} and \cite{CGR}. In \cite{Sch} it was shown that Berezin-Toeplitz quantization on arbitrary compact K\"ahler manifolds gives rise to a star product which was later identified in \cite{KSch} as the opposite product of a star product with separation of variables whose parameterizing form $\omega$ was calculated. It was a very difficult task to give explicit formulas for star products with separation of variables which were known only in particular cases (see \cite{BBEW}, \cite{AL}). In \cite{RT} an explicit formula for a non-normalized star product with separation of variables on an arbitrary K\"ahler manifold was given in terms of directed graphs. Gammelgaard gave in \cite{G} a remarkable explicit formula for a general star product with separation of variables on an arbitrary K\"ahler manifold in terms of directed graphs without cycles. In \cite{Xu1} an explicit formula for Berezin star product was given in terms of strongly connected directed graphs. Then in \cite{Xu2} this formula was used to prove Gammelgaard's formula for Berezin-Toeplitz star product by inverting the formal Berezin transform.

In this paper we give an alternative proof of Gammelgaard's formula. We construct an algebra of  operators on a formal Fock space expressed in terms of directed acyclic graphs and show that Gammelgaard's formula is equivalent to an inversion formula for an operator in that algebra.  We prove a composition formula for that algebra using the formalism of pre-Feynman and Feynman diagrams developed in \cite{A} and derive the inversion formula. Our approach gives more insight  into the question why the directed graphs in Gammelgaard's formula have no cycles.

\section{A Weyl commutation relation}

 We recall the construction of a general star product with separation of variables on an arbitrary pseudo-K\"ahler manifold from \cite{CMP1}. Let  $(M, \omega_{-1})$  be a pseudo-K\"ahler manifold of complex dimension $m$ and let
\[
      \omega = \frac{1}{\nu}\omega_{-1} + \omega_0 + \nu \omega_1 + \ldots
\]
be a formal closed (1,1)-form on $M$. We call it (somewhat loosely) a formal deformation of the pseudo-K\"ahler form $\omega_{-1}$. On an arbitrary contractible coordinate chart $(U, \{z^k, \bar z^l\})$ on $M$ the closed (1,1)-forms $\omega_r$ have potentials $\Phi_r$ such that $\omega_r = i\p\bar\p \Phi_r$. The formal function
\[
      \Phi: = \frac{1}{\nu} \Phi_{-1} +  \Phi_0 + \nu  \Phi_1 + \ldots
\]
is a potential of $\omega$. It was proved in \cite{CMP1} that there exists a unique star product with separation of variables on $U$ such that
\[
      L_{\frac{\p\Phi}{\p z^k}} = \frac{\p\Phi}{\p z^k} + \frac{\p}{\p z^k} \mbox{ and }  R_{\frac{\p\Phi}{\p \bar z^l}} = \frac{\p\Phi}{\p \bar z^l} + \frac{\p}{\p \bar z^l}.
\]
This star product does not depend on the choice of the potential $\Phi$ and of local holomorphic coordinates. There is a unique global star product with separation of variables $\ast_\omega$ on $M$ that agrees with this star product on every contractible coordinate chart. Every  star product with separation of variables on $M$ is of the form $\ast_\omega$ for a uniquely determined $\omega$. We say that $\ast_\omega$ is parameterized by $\omega$. In what follows we fix $\omega$ and drop the subscript $\omega$ in $\ast_\omega$.

Let $(U, \{z^k, \bar z^l\})$ be a  contractible coordinate chart  on $M$ and let $\ast$ be the star product with separation of variables on $U$ parameterized by a formal form $\omega$. Denote by $\F(U)$ the star algebra on $U$,
\[
     \F(U) = (C^\infty(U)[\nu^{-1},\nu]],\ast),
\]
by $\L(U)$ the algebra of left star multiplication operators on $U$, and by $\D(U)$ the algebra of all $\nu$-formal differential operators on $U$, i.e, the operators of the form
\[
    A = \sum_{r \geq n} \nu^r A_r,
\] 
where $n \in \Z$ and $A_r$ is a differential operator on $U$. It was proved in \cite{CMP1} that the algebra $\L(U)$ is the subalgebra of $\D(U)$ consisting of all formal differential operators on $U$ which commute with the operators
\[
     R_{\bar z^l} = \bar z^l  \mbox{ and }  R_{\frac{\p\Phi}{\p \bar z^l}} = \frac{\p\Phi}{\p \bar z^l} + \frac{\p}{\p \bar z^l},
\]
where $\Phi$ is a potential of $\omega$. The algebra $\L(U)$ is isomorphic to $\F(U)$ via the isomorphism 
\begin{equation}\label{E:isomaaone}
\L(U) \ni A \mapsto A\unit,
\end{equation}
where $A\unit$ denotes the operator $A$ applied to the unit constant function~$\unit$. The inverse isomorphism is $\F(U) \ni f \mapsto L_f$. We have $L_f \unit = f \ast \unit = f$.

We introduce formal variables $\eta^k$ and $\zeta_k, 1 \leq k \leq m$, and consider the algebras of formal series $\L(U)[[\eta,\zeta]]$ and $\F(U)[[\eta,\zeta]]$,
where the multiplication is extended from $\L(U)$ and $\F(U)$, respectively, by $\eta$- and $\zeta$-linearity. Then $\L(U)[[\eta,\zeta]]$ is the algebra of left multiplication operators of the algebra $\F(U)[[\eta,\zeta]]$, and the mapping (\ref{E:isomaaone}) extends to an isomorphism of $\L(U)[[\eta,\zeta]]$ onto $\F(U)[[\eta,\zeta]]$.

An element $f(\nu,\eta,\zeta) \in \F(U)[[\eta,\zeta]]$ is given by a formal series
\[
   f(\nu,\eta,\zeta) = \sum_{s \in \Z,n \geq 0,t \geq 0} \nu^s f_{s, k_1 \ldots k_n}^{p_1 \ldots p_t}(z,\bar z) \eta^{k_1}\ldots \eta^{k_n} \zeta_{p_1} \ldots \zeta_{p_t},  
\]
where the coefficients $ f_{s, k_1 \ldots k_n}^{p_1 \ldots p_t}(z,\bar z)$ are separately symmetric in the indices $k_i$ and $p_i$ and where for fixed $n$ and $t$ only finitely many coefficients $ f_{s, k_1 \ldots k_n}^{p_1 \ldots p_t}(z,\bar z)$ with negative $s$ are nonzero. Otherwise speaking, for each $ f(\nu,\eta,\zeta)$ there exists a function $\kappa(n,t)$ such that $ f_{s, k_1 \ldots k_n}^{p_1 \ldots p_t}(z,\bar z) =0$ for $s < \kappa(n,t)$.

We introduce operators
\[
     Z = \zeta_p z^p \mbox{ and } H = \eta^k\left( \frac{\p\Phi}{\p z^k} + \frac{\p}{\p z^k}\right)
\]
in $\L(U)[[\eta,\zeta]]$. Their commutator
\begin{equation}\label{E:commut}
     [H,Z] = \eta^k \zeta_k
\end{equation}
lies in the center of the algebra $\L(U)[[\eta,\zeta]]$. Since both $Z$ and $H$ are in the ideal $\langle\eta,\zeta\rangle$  generated by $\eta^k$ and $\zeta_p$, the exponential series $\exp\{Z\}$ and $\exp\{H\}$ converge in the $\langle\eta,\zeta\rangle$-adic topology and lie in $\L(U)[[\eta,\zeta]]$.
\begin{lemma}
The operators $e^Z$ and $e^H$ satisfy the Weyl commutation relation
\begin{equation}\label{E:weyl}
     e^H e^Z = e^{\eta^k\zeta_k} e^Z e^H
\end{equation}
in the algebra $\L(U)[[\eta,\zeta]]$.
\end{lemma}
\begin{proof}
We have, using (\ref{E:commut}),
\begin{eqnarray*}
     e^H e^Z e^{-H} = \exp\left\{ e^H Z e^{-H}   \right\} = \exp\left\{ e^{\ad(H)} Z   \right\} =\\
 \exp\left\{Z + \eta^k\zeta_k   \right\} = e^{\eta^k\zeta_k} e^Z.
\end{eqnarray*}
\end{proof}
Given a function $f(z)$, we denote by $f(z+\eta)$ the formal Taylor series
\[
        f(z+\eta) = e^{\eta^k \frac{\p}{\p z^k}}f(z).
\]
\begin{lemma}\label{L:ehunit}
   We have
\[
         e^H\unit =  e^{\Phi(z+\eta,\bar z)-\Phi(z,\bar z)}. 
\]
\end{lemma}
\begin{proof}
Consider the system of differential equations
\begin{equation}\label{E:syst}
     \frac{\p f}{\p \eta^k} = \left( \frac{\p\Phi}{\p z^k} + \frac{\p}{\p z^k}\right)f, \ 1 \leq k \leq m,
\end{equation}
with the initial condition $f|_{\eta=0}=\unit$ on the space $\F(U)[[\eta]]$. Since  the operators
\[
     \frac{\p\Phi}{\p z^k} + \frac{\p}{\p z^k}, \ 1 \leq k \leq m,
\]
pairwise commute, the system (\ref{E:syst}) has a unique solution 
\[
f = e^{H}\unit =\exp\left\{\eta^k \left( \frac{\p\Phi}{\p z^k} + \frac{\p}{\p z^k}\right)\right\}\unit .
\]
The formal series $ g := \exp\{\Phi(z+\eta,\bar z)-\Phi(z,\bar z)\}$ satisfies the initial condition $g|_{\eta=0}= \unit$. Since
\[
    \frac{\p g}{\p z^k} = \left( \frac{\p \Phi(z+\eta,\bar z)}{\p z^k} -  \frac{\p \Phi(z,\bar z)}{\p z^k}\right)g,
\]
we get that $g$ satisfies the system  (\ref{E:syst}):
\[
     \frac{\p g}{\p \eta^k} = \frac{\p \Phi(z+\eta,\bar z)}{\p z^k}g=  \left( \frac{\p\Phi}{\p z^k} + \frac{\p}{\p z^k}\right)g.
\]
Therefore, the formal series $f = \exp\{H\}\unit$ and $g$ are equal.
\end{proof}
Applying the isomophism (\ref{E:isomaaone}) to both sides of the Weyl commutation relation (\ref{E:weyl}), we obtain an equality in the algebra $\F(U)[[\eta,\zeta]]$,
\begin{equation}\label{E:prodexp}
e^{\Phi(z+\eta,\bar z)-\Phi(z,\bar z)} \ast e^{\zeta_k z^k} =
e^{\zeta_k\eta^k} e^{\zeta_k z^k} e^{\Phi(z+\eta,\bar z)-\Phi(z,\bar z)}.
\end{equation}
Introduce a family of commuting operators in $\D(U)[[\eta]]$,
\begin{eqnarray}\label{E:xioper}
    \Xi_l = e^{-(\Phi(z+\eta,\bar z)-\Phi(z,\bar z))}\frac{\p}{\p \bar z^l} e^{\Phi(z+\eta,\bar z)-\Phi(z,\bar z)}=\\
\frac{\p}{\p \bar z^l} +  \frac{\p (\Phi(z+\eta,\bar z)-\Phi(z,\bar z))}{\p \bar z^l}, \ 1 \leq l \leq m. \nonumber
\end{eqnarray}
Let $\bar \eta^l,  \ 1 \leq l \leq m$, be formal antiholomorphic variables. The formal form $\omega$ determines the formal Calabi function 
\[
D(\eta,\bar\eta) = \Phi(z+\eta,\bar z + \bar \eta) -
\Phi(z+\eta,\bar z) - \Phi(z, \bar z + \bar \eta) + \Phi(z,\bar
z)
\]
in $\F(U)[[\eta,\bar\eta]]$. It does not depend on the choice of the potential $\Phi$. Observe that $D(\eta,\bar\eta)$ lies in the ideal $\langle \eta \bar\eta\rangle$ generated by the products $\eta^k\bar \eta^l$. We write formally $D(\eta,0)=0$ and $D(0, \bar\eta)=0$. 
The following lemma can be proved along the same lines as Lemma \ref{L:ehunit}.
\begin{lemma}\label{L:exiunit}
   We have
\[
         \exp\left\{\bar \eta^l \Xi_l\right\}\unit =  e^{D(\eta,\bar \eta)}. 
\]
\end{lemma}

\section{Tensor identities}

In what follows we call a tensor an indexed array of numbers defined at a point of a coordinate chart $U \subset \C^m$. We do not assume that such tensors determine coordinate-invariant geometric objects on $U$. We will use tensor notations (rather than equivalent multi-index notations) to comply with Feynman diagrams in the rest of the paper. 

Set $[m] = \{1,2,\ldots,m\}$ and denote by $\mathbf{I}$ the disjoint union of Cartesian powers of the set $[m]$,
\[
     \mathbf{I} = \bigsqcup_{n\geq 0} [m]^n.
\]
A tensor index $I \in \mathbf{I}$ with $|I|=n$ is an $n$-tuple of natural numbers, $I=  i_1 \ldots i_n$, where $i_k \leq m$. A tensor $u$ of type $(p,k)$ is a function on $\mathbf{I}^{k+p}$ written as $u_{I_1\ldots I_k}^{J_1 \ldots J_p}$. Given tensors $u^I$ and $v_J$, their contraction is defined by the formula
\begin{equation}\label{E:contr}
   u^I v_I = \sum_{I \in \mathbf{I}} u^I v_I =  \sum_{n \geq 0} u^{i_1 \ldots i_n} v_{i_1 \ldots i_n}
\end{equation}
if the series on the right-hand side has a finite number of nonzero summands. If $u^I$ and $v_J$ are formal tensors, then the series on the right-hand side of (\ref{E:contr}) can be infinite but should have a finite number of nonzero summands at the same power of the formal parameter $\nu$.

We will work with tensors that are separately symmetric in the tensor indices and refer to tensors of types (2,0), (1,1), and (0,2), as to (infinite) matrices.

Introduce the following tensor $\Delta_K^I$ separately symmetric in $I$ and $K$, such that $\Delta_K^I =0$ for $|I| \neq |K|$ and
\begin{equation}\label{E:delta}
      \Delta_{k_1\ldots k_n}^{i_1 \ldots i_n} =\frac{1}{n!} \sum_{\sigma \in S_n} \delta_{k_{\sigma(1)}}^{i_1}\ldots \delta_{k_{\sigma(n)}}^{i_n},
\end{equation}
where $S_n$ is the symmetric group. If $f_I$ and $g^K$ are symmetric tensors, then $\Delta^I_K f_I = f_K$ and $\Delta_K^I g^K = g^I$, which means that $\Delta_K^I$ is the identity matrix.
We will use the following notations  for $K = k_1 \ldots k_n$,
\[
     \eta^K = \eta^{k_1} \ldots \eta^{k_n}, \zeta_K = \zeta_{k_1}\ldots \zeta_{k_n},  \mbox{ and }  \left(\frac{\p}{\p \eta}\right)^K = \frac{\p}{\p \eta^{k_1}} \ldots \frac{\p}{\p \eta^{k_n}},
\]
and similar notations for other variables.

A star product with separation of variables $\ast$ on a coordinate chart $(U, \{z^k,\bar z^l\})$ is given by the formula
\begin{eqnarray}\label{E:tensstar}
     f \ast g = \sum_{r\geq 0} \nu^r C_r^{LK} \left(\frac{\p}{\p \bar z}\right)^L f\left(\frac{\p}{\p z}\right)^K g = \\
 C^{LK} \left(\frac{\p}{\p \bar z}\right)^L f\left(\frac{\p}{\p z}\right)^K g, \nonumber
\end{eqnarray}
where the tensor $C_r^{LK}$ is separately symmetric in $K$ and $L$ and 
\[
     C^{LK} = \sum_{r\geq 0} \nu^r C_r^{LK}
\]
is a formal tensor. Formula (\ref{E:xioper}) implies that
\begin{equation}\label{E:xij}
     e^{-(\Phi(z+\eta,\bar z)-\Phi(z,\bar z))} \left(\frac{\p}{\p \bar z}\right)^L e^{\Phi(z+\eta,\bar z)-\Phi(z,\bar z)} = \Xi_L,
\end{equation}
where  for $L = l_1 \ldots l_n$ we assume that $\Xi_L = \Xi_{l_1} \ldots \Xi_{l_n}$. Observe that
\begin{equation}\label{E:zk}
    \left(e^{-\zeta_k z^k} \left(\frac{\p}{\p z}\right)^K e^{\zeta_k z^k}\right)\unit = \zeta_K.
\end{equation}
Using Eqns. (\ref{E:prodexp}),(\ref{E:tensstar}),(\ref{E:xij}), and (\ref{E:zk}), we obtain the following formula,
\begin{equation}\label{E:fund}
       C^{LK} \left(\Xi_L\unit\right) \zeta_K = e^{\eta^k\zeta_k}.
\end{equation}

It is easy to check that
\begin{equation}\label{E:easy}
   \frac{1}{|I|!} \left(\frac{\p}{\p\zeta}\right)^I \zeta_K \bigg|_{\zeta=0} = \Delta_K^I.
\end{equation}
If a tensor $A_K^I$ is separately symmetric in $I$ and $K$, then $A_K^I \eta^K =\eta^I$ if and only if $A_K^I = \Delta_K^I$. Applying the operator $(\p/\p \zeta)^I$ to both sides of (\ref{E:fund}) and setting $\zeta=0$, we obtain that
\begin{equation}\label{E:xieta}
      \left(|I|!\right) C^{LI} \left(\Xi_L\unit\right) = \eta^I.
\end{equation}
It follows from Lemma \ref{L:exiunit} that
\begin{equation}\label{E:fundred}
     \left(\frac{\p}{\p \bar\eta}\right)^L e^{D(\eta,\bar\eta)}\bigg|_{\bar\eta = 0} = \Xi_L\unit.
\end{equation}
Writing the formal series $\exp\{D(\eta,\bar\eta)\}$ explicitly,
\begin{equation}\label{E:edseries}
     e^{D(\eta,\bar\eta)} = \sum_{t \in \Z} \nu^t E_{t, KL} \eta^K \bar \eta^L =     E_{KL} \eta^K \bar \eta^L, 
\end{equation}
where 
\[
E_{KL} = \sum_{t \in \Z} \nu^t E_{t, KL}
\]
is a formal tensor, we get from (\ref{E:fundred}) and (\ref{E:edseries}) that
\begin{equation}\label{E:etaxi}
        \left(|L|!\right) E_{KL}  \eta^K = \Xi_L\unit.
\end{equation}
Now, it follows from (\ref{E:easy}), (\ref{E:xieta}), and  (\ref{E:etaxi}) that
\begin{equation}\label{E:tensone}
     \left(|L|!\right)  E_{KL} C^{LI} \left(|I|!\right) =  \Delta_K^I .
\end{equation}
The summation over the infinite index set $\mathbf{I}$ in (\ref{E:tensone})  is such that for fixed tensor indices $I$ and $K$, at each power of the formal parameter $\nu$ only finitely many summands are nonzero.
Starting with the operators
\[
     \bar Z = \bar\zeta_q\bar z^q \mbox{ and } \bar H = \bar\eta^l\left( \frac{\p\Phi}{\p \bar z^l} + \frac{\p}{\p \bar z^l}\right)
\]
instead of $Z$ and $H$, one can prove along the same lines that
\begin{equation}\label{E:tenstwo}
     \left(|J|!\right)   C^{JK} E_{KL} \left(|K|!\right) =  \Delta_L^J .
\end{equation}
The tensor identities (\ref{E:tensone}) and (\ref{E:tenstwo}) are valid at any point of $U$. They mean that the infinite matrices
\[
      \left(|L|!\right)   C^{LK}  \left(|K|!\right) \mbox{ and } E_{KL}
\]
are inverse to each other. In the next section we will interpret these identities at a point of $U$ as inversion formulas for two operators on a formal Fock space. We summarize the results of this section in the following proposition.
\begin{proposition}\label{P:summary}
  Given a formal deformation $\omega$ of a pseudo-K\"ahler form $\omega_{-1}$ on a coordinate chart $U \subset \C^m$ with the corresponding formal Calabi function $D(\eta,\bar\eta)$, the tensor coefficients $C^{LK}$ of the star product with separation of variables $\ast_\omega$ and the tensor coefficients $E_{KL}$ of the function 
$\exp\{D(\eta,\bar\eta)\}$ satisfy the identities (\ref{E:tensone}) and (\ref{E:tenstwo}).
\end{proposition}

\noindent{\it Remark.} The coefficients $C^{LK}$ and $E_{KL}$ do not behave as tensors under holomorphic coordinate changes. However, it is convenient to call them tensor coefficients.

Now we want to illustrate the calculations done in this section using the simplest star product with separation of variables,
the anti-Wick star product on $\C^m$. It is parameterized by the formal form $\omega = \frac{i}{\nu}  g_{kl} dz^k \land d \bar z^l$ with the potential 
$\Phi = \frac{1}{\nu}g_{kl}z^k \bar z^l$, where $g_{kl}$ is a nondegenerate matrix with
constant coefficients. It is known (see \cite{CMP1}) that the
anti-Wick star product is given by the following explicit formula:
\begin{equation}\label{E:awick}
f \ast g = \sum_{r=0}^\infty \frac{\nu^r}{r!} g^{l_1 k_1}\ldots
g^{l_r k_r}\frac{\p^r f}{\p \bar z^{l_1}\ldots \p \bar z^{l_r}}\frac{\p^r g}{\p z^{k_1}\ldots \p z^{k_r}},
\end{equation}
where $g^{lk}$ is the inverse matrix of $g_{kl}$. This product can be written in the tensor form as (\ref{E:tensstar}), where the formal tensor $C^{LK}$ separately symmetric in $K$ and $L$ is such that $C^{LK} = 0$ if $|K| \neq |L|$ and 
\begin{equation}\label{E:wickc}
       C^{l_1 \ldots l_r k_1 \ldots k_r} = \frac{\nu^r}{(r!)^2}\sum_{\sigma \in S_r} g^{l_1 k_{\sigma(1)}}\ldots g^{l_r k_{\sigma(r)}}.
\end{equation}
The formal Calabi function $ D(\eta,\bar \eta)$ corresponding to the potential $\Phi = \frac{1}{\nu}g_{kl}z^k \bar z^l$ is
\[
     D(\eta,\bar \eta) = \frac{1}{\nu}g_{kl}\eta^k \bar \eta^l.
\]
Writing $\exp\{D(\eta,\bar \eta)\}$ in the tensor form,
\[
      e^D(\eta,\bar \eta) = E_{KL} \eta^K \bar\eta^L,
\]
we see that the formal tensor $ E_{KL}$ separately symmetric in $K$ and $L$ is such that $ E_{KL} =0$ if $|K| \neq |L|$ and 
\begin{equation}\label{E:wicke}
    E_{k_1 \ldots k_r l_1\ldots l_r} = \frac{1}{\nu^r (r!)^2}\sum_{\sigma \in S_r} g_{k_1 l_{\sigma(1)}}\ldots g_{k_r l_{\sigma(r)}}.
\end{equation}
Formulas (\ref{E:tensone}) and (\ref{E:tenstwo}) for the anti-Wick star product can be verified directly.

\section{Operators on a formal Fock space}

In this section we list elementary properties of formal Fock spaces that are used in the proofs in the rest of the paper.

The space of formal symmetric tensors $f_K = \sum_{s=0}^\infty \nu^s f_{s,K}$ can be identified with the formal Fock space $\V = \C[[\nu, \eta^1, \ldots \eta^m]]$. A tensor $f=f_K$ corresponds to the formal series $f = \nu^s f_{s, K}\eta^K$.  
We will treat $\V$ as a commutative algebra over $\C$ and denote by $\I =
\langle \nu, \eta^1, \ldots \eta^m \rangle$  the ideal in $\V$
generated by $\nu$ and $\eta^i$. We endow $\V$ with the $\I$-adic
topology. For $k \in \N$ denote by $\pi_k: \V \to \V/\I^k$ the quotient
mapping. The vector space $\V/\I^k$ is finite dimensional and
discrete in the quotient topology.

If $A: \V \to \V$ is a continuous $\C$-linear mapping, then for
any $k\in \N$ there exists $N \in \N$ such that $\I^N$ lies in
the kernel of the mapping $\pi_k \circ A: \V \to \V/\I^k$. Thus $A$
induces a mapping $\tilde A: \V/\I^N \to \V/\I^k$ such that $\pi_k
\circ A = \tilde A \circ \pi_N$. This implies that a continuous $\nu$-linear (i.e., commuting with the multiplication by $\nu$) operator $A$ on $\V$ can be given in terms of a formal tensor 
\[
A_K^I = \sum_{s = 0}^\infty \nu^s A_{s, K}^I
\]
separately symmetric in $I$ and $K$ and such that for fixed $s$ and $K$ there are finitely many nonzero entries $ A_{s, K}^I$. For $f(\nu, \eta) \in \V$, we set
\[
      (Af)(\nu,\eta) = \frac{1}{|I|!}A_K^I \eta^K \left(\frac{\p}{\p\lambda}\right)^I  f(\nu, \lambda)\bigg|_{\lambda=0}.
\]
If $f$ is written in the tensor form as $f = f_I \eta^I$, where $f_I$ is a symmetric formal tensor, then $(Af)_K = A_K^I f_I$. We will say that $A_K^I$ is the matrix of the operator $A$.
The identity operator corresponds to the tensor $\Delta_K^I$ and can be written explicitly as
\begin{equation}\label{E:ident}
     f(\nu, \eta) \mapsto \exp\left\{\eta^k \frac{\p}{\p \lambda^k}\right\} f(\nu, \lambda)\bigg |_{\lambda=0} = f(\nu, \lambda+\eta)\bigg |_{\lambda=0} = f(\nu, \eta).
\end{equation}
 Now assume that $(M,\omega_{-1})$ is a pseudo-K\"ahler manifold with the  pseudo-K\"ahler metric tensor  $g_{kl}$ and $g^{lk}$ is the corresponding K\"ahler-Poisson tensor. We introduce two formal tensors, $G_{KL}$ and $G^{LK}$, separately symmetric in $K$ and $L$ and such that $G_{KL}=G^{LK}=0$ if $|K| \neq |L|$,
\begin{equation}\label{E:gkl}
   G_{k_1 \ldots k_r l_1\ldots l_r} = \frac{1}{\nu^r (r!)^2}\sum_{\sigma \in S_r} g_{k_1 l_{\sigma(1)}}\ldots g_{k_r l_{\sigma(r)}},
\end{equation}
given by the same formula as (\ref{E:wicke}) but with $g_{kl}$ not necessarily constant, and
\[
      G^{l_1 \ldots l_r k_1 \ldots k_r} = \nu^r\sum_{\sigma \in S_r} g^{l_1 k_{\sigma(1)}}\ldots g^{l_r k_{\sigma(r)}},
\]
which differs from the right-hand side of  (\ref{E:wickc}) by a factor of  $(r!)^2$. It can be checked directly as in the case of formulas (\ref{E:tensone}) and (\ref{E:tenstwo}) for the anti-Wick star product that the matrices  $G_{KL}$ and $G^{LK}$ are inverse to each other, 
\[
G_{KL}G^{LI} = \Delta_K^I \mbox{ and } G^{JK}G_{KL} = \Delta_L^J.
\]
Let $\ast$ be a star product with separation of variables on $(M,\omega_{-1})$  corresponding to a formal deformation $\omega$ of the form $\omega_{-1}$. Fix a coordinate chart $(U, \{z^k,\bar z^l\})$  on $M$. According to Proposition \ref{P:summary}, formulas (\ref{E:tensone}) and (\ref{E:tenstwo}) hold for the tensor coefficients $C^{LK}$ of the star product $\ast$ and the tensor coefficients $E_{KL}$ of the function $\exp\{D(\eta,\bar\eta)\}$, where $D(\eta,\bar\eta)$ is the formal Calabi function corresponding to $\omega$. These formulas mean that the matrices $(|L|!) C^{LK}(|K|!)$ and $E_{KL}$ are inverse to each other.

We introduce a formal tensor $C_K^I  = \sum_s \nu^s C_{s, K}^I$ given by the formula 
\begin{equation}\label{E:cki}
   C_K^I = G_{KL}(|L|!) C^{LI}(|I|!).
\end{equation}
Observe that for the anti-Wick star product  $C_K^I = \Delta_K^I$. For $I= i_1 \ldots i_p$ and $K =k_1 \ldots k_r$, we write both sides of (\ref{E:cki}) explicitly as series in the formal parameter $\nu$:
\begin{eqnarray*}
   \sum_s \nu^s C_{s, k_1 \ldots k_r}^{i_1 \ldots i_p} = \frac{1}{\nu^r (r!)^2}\sum_{\sigma \in S_r} g_{k_1 l_{\sigma(1)}}\ldots g_{k_r l_{\sigma(r)}} \sum_{t \geq 0} \nu^{t}r! p! C_{t}^{l_1 \ldots l_r i_1 \ldots i_p}=\\
\sum_s \nu^s p! g_{k_1 l_1}\ldots   g_{k_r l_r}  C_{s+r}^{l_1 \ldots l_r i_1 \ldots i_p},
\end{eqnarray*}
whence
\begin{equation}\label{E:cski}
        C_{s, k_1 \ldots k_r}^{i_1 \ldots i_p}  =  p!\, g_{k_1 l_1}\ldots   g_{k_r l_r}  C_{s+r}^{l_1 \ldots l_r i_1 \ldots i_p}.
\end{equation}
It was proved in \cite{BW},\cite{N}, and \cite{CMP3} that deformation quantizations with separation of variables are natural in the sense of \cite{GR}. Naturality of a star product (\ref{E:star}) means that the bidifferential operator $C_r(\cdot,\cdot)$ is of order not greater than $r$ in each argument. Now (\ref{E:cski}) implies that $s+r \geq r$ and $s+r \geq p$, which means that the component $C_{s,K}^I$ is nonzero only if $s \geq  0$ and that for fixed $s$ and $K$ only finitely many entries $C_{s,K}^I$ are nonzero. Therefore, the formal tensor $C_K^I$ is the matrix of a continuous operator $C$ on the formal Fock space $\V$. Next we introduce a formal tensor $E_K^I = \sum_s \nu^s E_{s, K}^I$ given by the formula
\begin{equation}\label{E:eki}
      E_K^I =E_{KL} G^{LI}.
\end{equation}
Observe that for the anti-Wick star product  $E_K^I = \Delta_K^I$. For $I= i_1 \ldots i_p$ and $K =k_1 \ldots k_r$, we write both sides of (\ref{E:eki}) explicitly as series in the formal parameter $\nu$:
\begin{eqnarray*}
    \sum_s \nu^s E_{s, k_1 \ldots k_r}^{i_1 \ldots i_p} = \sum_t  \nu^t E_{t, k_1 \ldots k_r l_1 \ldots l_p} \nu^p \sum_{\sigma \in S_p} g^{l_1  i_{\sigma(1)}} \ldots g^{ l_p  i_{\sigma(p)}}=\\
\sum_s \nu^s p!  E_{s-p, k_1 \ldots k_r l_1 \ldots l_p} g^{l_1  i_1} \ldots g^{l_p  i_p},
\end{eqnarray*}
whence
\begin{equation}\label{E:eski}
      E_{s, k_1 \ldots k_r}^{i_1 \ldots i_p} =  p!\,  E_{s-p, k_1 \ldots k_r l_1 \ldots l_p} g^{l_1  i_1} \ldots g^{l_p  i_p}.
\end{equation}
Let us write the formal Calabi function 
\[
D(\eta,\bar \eta) = \frac{1}{\nu}D_{-1}(\eta,\bar \eta) + D_0(\eta,\bar \eta) + \nu D_1(\eta,\bar \eta) + \ldots
\]
corresponding to $\omega$ in the tensor form,
\[
     D(\eta,\bar \eta) = D_{KL}\eta^K \bar\eta^L = \sum_{r =-1}^\infty \nu^r D_{r, KL}\eta^K \bar\eta^L.
\]
Since $D(\eta,\bar \eta)$ lies in the ideal $\langle \eta \bar\eta\rangle$ generated by the products $\eta^k\bar\eta^l$, an entry $D_{r, KL}$ is nonzero only if $r + |K| \geq 0$ and $r+ |L| \geq 0$. Therefore, a tensor coefficient  $E_{r, KL}$ of $\exp\{D(\eta,\bar \eta)\}$ is nonzero also only if $r + |K| \geq 0$ and $r+ |L| \geq 0$. It follows from (\ref{E:eski}) that an entry $E_{s,K}^I$ is nonzero only if $s \geq 0$ and $s-|I| + |K| \geq 0$. Therefore, for fixed $s$ and $K$ only finitely many entries $E_{s,K}^I$ are nonzero, which means that the  formal tensor $E_K^I$ is the matrix of a continuous operator $E$ on the formal Fock space $\V$. Formulas (\ref{E:tensone}) and (\ref{E:tenstwo}) imply that the matrices $C_K^I$ and $E_K^I$ are inverse to each other. We have thus proved the following proposition.
\begin{proposition}\label{P:operce}
 Given a star product with separation of variables on a pseudo-K\"ahler manifold parameterized by a formal form $\omega$,  the formal tensors $C_K^I$ and $E_K^I$ on a coordinate chart $(U, \{z^k,\bar z^l\})$ expand in formal series in nonnegative powers of the formal parameter $\nu$,
\[
     C_K^I = \sum_{s \geq 0} \nu^s C_{s,K}^I  \mbox{ and } E_K^I = \sum_{s \geq 0} \nu^s E_{s,K}^I.
\]
Moreover, for fixed $s$ and $K$ there are finitely many nonzero entries $ C_{s,K}^I$ and $ E_{s,K}^I$, which means that at a point of $U$ the tensors $C_K^I$ and $E_K^I$ are the matrices of continuous operators $C$ and $E$ on the formal Fock space $\V$ that are inverse to each other. For the anti-Wick star product both $C$ and $E$ are equal to the identity operator.
\end{proposition}

\section{Gammelgaard's formula}

In this section we introduce the ingredients used in Gammelgaard's formula following \cite{G} with minor modifications. All graphs in this paper are directed multigraphs (i.e., with possibly multiple edges) with no cycles. A directed graph is a set of vertices connected by directed edges. We will consider graphs with one source vertex which has only outgoing edges (tails) and one sink that has only incoming edges (heads). The source and the sink are called external vertices. The other vertices are called internal. A graph may have no internal vertices. We define a set of types of internal vertices 
\begin{equation}\label{E:types}
      \mathbf{T} = \{(p,q,r) \in \Z^3| p \geq 1, q \geq 1, r \geq -1, p+q+r \geq 2\}.
\end{equation}
Each internal vertex of type $(p,q,r)\in \mathbf{T}$ has $p$ incoming edges, $q$ outgoing edges, and weight $r$. An isomorphism of two graphs is a bijective mapping of vertices to vertices and edges to edges which preserves the types of the internal vertices and the way vertices are connected by edges. In particular, the source and the sink of one graph are mapped to the source and the sink of the other graph, respectively. We will denote by $\A$ the set of isomorphism classes of directed acyclic weighted graphs with one source and one sink. For brevity sake, we will say that a graph $\Gamma$  whose isomorphism class $[\Gamma]$  is  in $\A$ is a graph from $\A$.

We recall the definitions of the partition function and the weight of a graph $\Gamma$ from $\A$, introduced in \cite{G}. Let $\ast$ be a deformation quantization with separation of variables on a pseudo-K\"ahler manifold $(M,\omega_{-1})$ with the characterizing form $\omega$,
\[
     \Phi = \frac{1}{\nu}\Phi_{-1} + \Phi_0 + \nu \Phi_1 + \ldots
\]
be a potential of $\omega$ on a contractible holomorphic coordinate chart $(U, \{z^k,\bar z^l\})$, and  $\Gamma$ be a graph from $\A$.  To each internal vertex of type $(p,q,r)$ of $\Gamma$ Gammelgaard relates the tensor
\begin{equation*}
                 -\frac{\p^{p+q}\Phi_r}{\p z^{k_1}\ldots \p z^{k_p}  \p \bar z^{l_1}\ldots \p \bar z^{l_q}}.
\end{equation*}
The source with $\qq$ outgoing edges and the sink with $\pp$ incoming edges correspond to the tensors
\begin{equation*}
      \frac{\p^\qq f}{ \p \bar z^{l_1}\ldots \p \bar z^{l_\qq}} \mbox{ and } \frac{\p^\pp g}{\p z^{k_1}\ldots \p z^{k_\pp}},
\end{equation*}
respectively. 
The tensors corresponding to two vertices connected by an edge are contracted by the tensor $g^{lk}$ so that the tail corresponds to the antiholomorphic index $l$ and the head corresponds to the holomorphic index $k$. The resulting function is the partition function of the graph $\Gamma$ introduced in \cite{G}. We denote it by $G_\Gamma(f,g)$. 

We define the type of a graph $\Gamma$ as a triple $(\pp,\qq,\m)$, where $\pp$ and $\qq$ are the numbers of edges incident to the sink and the source, respectively, and $\m$ is the multiplicity function of internal vertices of $\Gamma$, i.e., $\m(p,q,r)$ is the multiplicity of vertices of type $(p,q,r)$ in $\Gamma$. The data $\pp,\qq,$ and $\m$ are not independent. The number of heads in $\Gamma$  is equal to the number of tails,
\begin{equation}\label{E:nedges}
      \pp + \sum_{(p,q,r)\in \mathbf{T}} p\, \m(p,q,r) = 
 \qq + \sum_{(p,q,r)\in \mathbf{T}} q\, \m(p,q,r).
\end{equation}
Denote this number by $N(\Gamma)$. It is also the number of edges in $\Gamma$.

We define the weight of a graph $\Gamma$ as follows.  The weight of each edge is one, the weight of an  internal vertex of type $(p,q,r)$ is $r$, and the weight of an external vertex is zero. The weight $W(\Gamma)$ of a graph $\Gamma$ is the sum of weights of all vertices and edges,
\[
      W(\Gamma) = N(\Gamma) + \sum_{(p,q,r)\in \mathbf{T}} r\,\m(p,q,r).
\]
 The partition function and the weight of a graph $\Gamma$ depend only on its isomorphism class $[\Gamma] \in \A$.

{\it Example.} The type of a graph  $\Gamma$ with one internal vertex of type $(1,2,3)$,
\begin{equation}\label{E:graph}
 { \xygraph{
!{<0cm,0cm>;<1cm,0cm>:<0cm,0cm>::}
!{(0,0) }*+{\circ_{\mathbf{in}}}="a"
!{(1,0) }*+{\bullet_{3}}="b"
!{(2.5,0) }*+{\circ_{\mathbf{out}}}="c"
"a":"b"
"b":@/^/"c"
"b":@/_/"c"
} }, 
\end{equation}
is $(2,1,\m)$, where $\m(1,2,3) =1$ and $\m(p,q,r)=0$ for all other vertex types.  The partition function $G_\Gamma(f,g)$ of the graph $\Gamma$ is
\[
      G_\Gamma(f,g) = -\frac{\p f}{\p \bar z^{l_1}} g^{l_1k_1} \frac{\p^3 \Phi_3}{\p z^{k_1} \p \bar z^{l_2} \p \bar z^{l_3}}g^{l_2k_2}g^{l_3k_3} \frac{\p^2 g}{\p z^{k_2} \p z^{k_3}}
\]
and the weight of $\Gamma$ is $W(\Gamma) = 6$. The group of automorphisms of $\Gamma$ is $Aut(\Gamma) = \Z_2$. It permutes the two edges connecting the internal vertex with the sink.

Gammelgaard discovered in \cite{G} a beautiful explicit formula for the star product $\ast$,
\begin{equation}\label{E:gammel}
     f \ast g = \sum_{[\Gamma] \in \A} \frac{\nu^{W(\Gamma)}}{|Aut(\Gamma)|} G_\Gamma(f,g),
\end{equation}
where $|Aut(\Gamma)|$ is the order of the automorphism group $Aut(\Gamma)$ of a graph $\Gamma$ and the summation is over a set of representatives of the isomorphism classes $\A$.

Denote by $\Lambda_n$ the graph from $\A$ with no internal vertices and $n$ edges. We have $W(\Lambda_n) = n$ and $Aut(\Lambda_n) = S_n$, the symmetric group of degree $n$. The partition functions of the graphs $\Lambda_0$ and $\Lambda_1$,
\[ { \xygraph{
!{<0cm,0cm>;<1.5cm,0cm>:<0cm,0cm>::}
!{(0,0) }*+{\circ_{\mathbf{in}}}="a"
!{(1,0) }*+{\circ_{\mathbf{out}}}="b"
} } \mbox{ and }
 { \xygraph{
!{<0cm,0cm>;<1.5cm,0cm>:<0cm,0cm>::}
!{(0,0) }*+{\circ_{\mathbf{in}}}="a"
!{(1,0) }*+{\circ_{\mathbf{out}}}="b"
"a":"b"
} },\]
are, respectively, $G_{\Lambda_0}(f,g) = fg = C_0(f,g)$ and
\[
   G_{\Lambda_1}(f,g) =g^{lk} \frac{\p f}{\p \bar z^l}\frac{\p g}{\p z^k} = C_1(f,g).
\]

The anti-Wick star product (\ref{E:awick}) is given by the Gammelgaard's formula
\begin{equation}\label{E:gamawick}
    f \ast g = \sum_{r=0}^\infty \frac{\nu^r}{r!} G_{\Lambda_r}(f,g).
\end{equation}
\section{Partition operators}\label{S:oper}

Let $\ast$ be a deformation quantization with separation of variables on a pseudo-K\"ahler manifold $(M,\omega_{-1})$ with the characterizing form $\omega$,
\[
     \Phi = \frac{1}{\nu}\Phi_{-1} + \Phi_0 + \nu \Phi_1 + \ldots
\]
be a potential of $\omega$ on a contractible holomorphic coordinate chart $(U, \{z^k,\bar z^l\})$, and  $\Gamma$ be a graph from $\A$ of type $(\pp,\qq,\m)$.  We will relate to $\Gamma$ a tensor $\hat\Gamma_K^I$ separately symmetric in $I$ and $K$ and such that $\hat\Gamma_K^I = 0$ unless $|I|=\qq$ and $|K| = \pp$ as follows.  To each internal vertex of type $(p,q,r)$ of $\Gamma$ we relate the tensor
\begin{equation*}
                 -\frac{\p^{p+q}\Phi_r}{\p z^{k_1}\ldots \p z^{k_p}  \p \bar z^{l_1}\ldots \p \bar z^{l_q}} g^{l_1i_1}\ldots  g^{l_qi_q}.
\end{equation*}
To each edge connecting the source directly with the sink we relate the Kronecker tensor $\delta_k^i$. To each edge connecting two internal vertices we relate a contraction of an upper index corresponding to the tail vertex with a lower index corresponding to the head vertex. Then we symmetrize the resulting tensor in the free lower and upper indices separately. The symmetrized tensor is denoted $\hat\Gamma_K^I$. The tensor $\hat\Gamma_K^I$ is the matrix of an operator $\hat\Gamma$ on the formal Fock space $\V$. This operator can be called the partition operator of the graph $\Gamma$. Isomorphic graphs have equal partition operators.

{\it Example.} The partition operator $\hat \Gamma$ of the graph $\Gamma$ given by (\ref{E:graph}) is
\[
      (\hat \Gamma f)(\nu,\eta) = -\frac{1}{2}\eta^{k_1}\frac{\p^3 \Phi_3}{\p z^{k_1} \p \bar z^{l_1} \p \bar z^{l_2}}g^{l_1i_1}g^{l_2i_2} \frac{\p}{\p\lambda^{i_1}}\frac{\p}{\p\lambda^{i_2}} f(\nu,\lambda)\bigg|_{\lambda=0}.
\]

{\it Example.} The operator $\hat \Lambda_n$ projects a formal symmetric tensor $f_K$ onto its component with $|K|=n$. The operator
\begin{equation}\label{E:ident}
      \sum_{n=0}^\infty \hat \Lambda_n
\end{equation}
is the identity operator on $\V$ (see formulas (\ref{E:delta}), (\ref{E:ident}), and (\ref{E:gamawick})).

We define the operator weight of an internal vertex of type $(p,q,r)$ to be $q+r$. Since for $(p,q,r) \in \mathbf{T}$ the inequalities $q \geq 1$ and $r \geq -1$ hold, we have $q+r \geq 0$.  The operator weight $\hat W(\Gamma)$ of a graph $\Gamma$ is defined as the sum of operator weights of all internal vertices,
\begin{equation}\label{E:sweight}
      \hat W(\Gamma) = \sum_{(p,q,r)\in \mathbf{T}} (q+r)\m(p,q,r).
\end{equation}
It depends only on the type of the graph $\Gamma$. Clearly, $\hat W(\Gamma) \geq 0$. We will use also the notation $\hat W(\pp,\qq,\m)$ for the sum in (\ref{E:sweight}).

\begin{lemma}\label{L:ops}
Any infinite linear combination
\begin{equation}\label{E:ops}
      \sum_{[\Gamma]\in\A} \nu^{\hat W(\Gamma)} c_\Gamma \hat \Gamma,
\end{equation}
where $c_\Gamma$ is a complex-valued function on $\A$, determines a continuous operator on the formal Fock space $\V$.
\end{lemma}
\begin{proof}

Let $\Gamma$ be a graph of type $(\pp,\qq,\m)$. For any $k$ the composition 
\begin{equation}\label{E:overi}
\pi_k \left(\nu^{\hat W(\Gamma)} c_\Gamma\hat \Gamma\right): \V \to \V/\I^k
\end{equation}
is zero if $\qq + \hat W(\Gamma) \geq k$. In order to prove the statement of the lemma we fix $\qq \geq 1$ and $s \geq 0$ and show that there are finitely many isomorphism classes of graphs in $\A$ of operator weight $s$ and with the source of degree  $\qq$. This implies that the composition (\ref{E:overi}) is nonzero only for finitely many summands in (\ref{E:ops}) and therefore  (\ref{E:ops}) gives a well-defined continuous operator on $\V$.

It follows from formulas  (\ref{E:nedges}) and (\ref{E:sweight}) that
\begin{equation}\label{E:diff}
     \qq + s = \pp + \sum_{(p,q,r)} (p+r)\m(p,q,r).
\end{equation}
 We have from (\ref{E:sweight}), (\ref{E:diff}), and the inequalities $p \geq 1, q \geq 1$, and $r \geq -1$ that
\[
    (q+r)\m(p,q,r) \leq s, (p+r)\m(p,q,r) \leq \qq + s, \mbox{ and } \pp \leq \qq + s.
\]
We see that if $\m(p,q,r) \geq 1$, then $p \leq \qq + s +1, q \leq s+1$, and $r \leq s-1$. It follows from the inequality $p+q+r \geq 2$ that $p+r \geq 1$ or $q+r \geq 1$, which implies that $\m(p,q,r) \leq \qq + s$. We have shown that the joint support of the multiplicity functions of all such graphs is finite and these multiplicity functions are uniformly bounded. Also, the degree $\pp$ of the sink of all these graphs has a common upper bound.  There exist finitely many isomorphism classes of such graphs. \end{proof}

Given a graph $\Gamma$ from $\A$, denote by $\pp(\Gamma)$ the degree of its sink. Gammelgaard's formula (\ref{E:gammel}) is equivalent to the following statement.
\begin{theorem}\label{T:operc}
 Given a star product with separation of variables $\ast$ on a pseudo-K\"ahler manifold $M$ parameterized by a formal form $\omega$, the operator $C$ of the star product $\ast$ given by the matrix (\ref{E:cki}) on a coordinate chart $U \subset M$ is expressed by the following explicit formula,
\begin{equation}\label{E:operc}
      C = \sum_{[\Gamma] \in \A} \frac{\nu^{\hat W(\Gamma)}\pp(\Gamma)!}{|Aut(\Gamma)|}\, \hat \Gamma.
\end{equation}
\end{theorem}
\begin{proof}
 Consider the contribution of the isomorphism class of a graph $\Gamma$ of type $(\pp,\qq, \m)$ from $\A$ to the tensor $C^{LK}$ of the star product $\ast$ in Gammelgaard's formula (\ref{E:gammel}) and to the right-hand side of (\ref{E:operc}).
First, write the partition function $G_\Gamma(f,g)$ in the tensor form,
\[
      G_\Gamma(f,g) = G_\Gamma^{LK}  \left(\frac{\p}{\p \bar z}\right)^L f\left(\frac{\p}{\p z}\right)^K g,
\]
where $G_\Gamma^{LK}$ is separately symmetric in $K$ and $L$. According to (\ref{E:cki}), in order to prove (\ref{E:operc}) we have to verify the equality
\begin{equation}\label{E:contrib}
       \frac{\nu^{\hat W(\Gamma)}\pp!}{|Aut(\Gamma)|}\, \hat \Gamma_K^I =G_{KL}\qq! \frac{\nu^{W(\Gamma)}}{|Aut(\Gamma)|} G_\Gamma^{LI}\pp!.
\end{equation}
By construction,
\[
    \hat \Gamma_{k_1 \ldots k_\qq}^{i_1 \ldots i_\pp} = g_{k_1 l_1}\ldots g_{k_\qq l_\qq} G_\Gamma^{l_1 \ldots l_\qq i_1 \ldots i_\pp}.
\]
It follows from (\ref{E:gkl}) that
\[
      \hat \Gamma_K^I = \nu^\qq \qq! G_{KL}G_\Gamma^{LI}.
\]
Now, (\ref{E:contrib}) immediately follows from the definitions of $W(\Gamma)$ and $\hat W(\Gamma)$.
\end{proof}
The fact that the operator $C$ continuous on the formal Fock space $\V$ is given by formula (\ref{E:operc}) agrees with Lemma \ref{L:ops}. 

Given a graph $\Gamma$ from $\A$, denote by $R(\Gamma)$ the number of internal vertices of $\Gamma$. Denote by $\U$ the set of equivalence classes of graphs from $\A$ such that their internal vertices are not connected by edges. Given a graph $\Gamma$ of type $(\pp,\qq, \m)$ from $\U$, its group of automorphisms $Aut(\Gamma)$ independently permutes the internal vertices of the same type, the outgoing edges of each internal vertex, the incoming edges of each internal vertex, and the edges directly connecting the source with the sink.  Thus,
\begin{equation}\label{E:autm}
    |Aut(\Gamma)| =D(\Gamma)! \prod_{(p,q,r) \in \mathbf{T}}\left(p! q!\right)^{\m(p,q,r)}\m(p,q,r)!,
\end{equation}
where 
\[
      D(\Gamma) = \qq - \sum_{(p,q,r) \in \mathbf{T}} p\, \m(p,q,r) = \pp - \sum_{(p,q,r) \in \mathbf{T}} q\, \m(p,q,r)
\]
is the number of edges directly connecting the source with the sink. 

For every graph type $(\pp,\qq,\m)$ there exists only one equivalence class of graphs in $\U$ of that type. It is obtained by connecting all internal vertices directly to the sink and the source and connecting the source and the sink directly by the remaining edges.

The operator $E$ on the formal Fock space $\V$ corresponding to a formal form $\omega$ can be expressed in terms of the  partition operators of the graphs from $\U$.
\begin{theorem}\label{E:opere}
Let $\omega$ be a formal deformation  of a pseudo-K\"ahler form $\omega_{-1}$ on a coordinate chart $U \subset \C^m$ and $D(\eta,\bar \eta)$ be the corresponding formal Calabi function, then the operator $E$ corresponding to the function $\exp\{D(\eta,\bar \eta)\}$ introduced in Proposition \ref{P:operce}  is given by the following explicit formula,
\begin{equation}\label{E:thm}
      E = \sum_{[\Gamma] \in \U} \frac{(-1)^{R(\Gamma)}\nu^{\hat W(\Gamma)}\pp(\Gamma)!}{|Aut(\Gamma)|} \hat \Gamma.
\end{equation}
\end{theorem}
\begin{proof}
Denote by $\mathbf{T}'$ the extension of the set $\mathbf{T}$ by the point $(1,1,-1)$,
\[
     \mathbf{T}' = \{(p,q,r) \in \Z^3| p \geq 1, q \geq 1, r \geq -1\}.
\]
We can write the formal Calabi function as an explicit series,
\[
    D(\eta,\bar\eta) = -\sum_{(p,q,r) \in \mathbf{T}'} \frac{\nu^r}{p!q!}F_{p,q,r},
\]
where
\begin{equation}\label{E:fact}
      F_{p,q,r} =  - \frac{\p^{p+q}\Phi_r}{\p z^{k_1} \ldots \p z^{k_p} \p \bar z^{l_1} \ldots \p \bar z^{l_q} }\eta^{k_1}\ldots \eta^{k_p} \bar \eta^{l_1} \ldots \bar\eta^{l_q}.
\end{equation}
We have by the multinomial theorem that
\begin{equation}\label{E:mult}
       \exp\{D(\eta,\bar \eta)\} = \sum_\n \prod_{(p,q,r) \in \mathbf{T}'} \frac{1}{\n(p,q,r)!} \left(\frac{(-1)\nu^r}{p!q!}F_{p,q,r}\right)^{\n(p,q,r)},
\end{equation}
where $\n(p,q,r)$ is the multiplicity function  of the factor  $F_{p,q,r}$.

There is a bijection between the set of all multiplicity functions $\n(p,q,r)$ on $\mathbf{T}'$ and the set of all graph types $(\pp,\qq,\m)$ such that
\begin{equation}\label{E:bij}
  \m = \n|_\mathbf{T}, \pp = \sum_{(p,q,r) \in \mathbf{T}'} q \n(p,q,r),
\mbox{ and } \qq = \sum_{(p,q,r) \in \mathbf{T}'} p \n(p,q,r).
\end{equation}
Fix a multiplicity function $\n(p,q,r)$ and the corresponding graph type $(\pp,\qq,\m)$. Denote by $\Gamma_\n$ a graph from $\U$ of type $(\pp,\qq,\m)$. Notice that the multiplicity of the factor $F_{1,1,-1} = - g_{kl} \eta^k \bar \eta^l$ coincides with the number of edges in $\Gamma_\n$ directly connecting the source and the sink, 
\[
           D(\Gamma_\n) = \n(1,1,-1).
\]
It follows from (\ref{E:autm}) that
\begin{equation}\label{E:aut}
    |Aut(\Gamma_\n)| =\prod_{(p,q,r) \in \mathbf{T}'}\left(p! q!\right)^{\n(p,q,r)}\n(p,q,r)!.
\end{equation}
Denote by $F^\n_{KL}$ the tensor separately symmetric in $K$ and $L$ and such that
\[
       F^\n_{KL} \eta^K \bar \eta^L = \prod_{(p,q,r) \in \mathbf{T}'} \left(F_{p,q,r}\right)^{\n(p,q,r)}.
\]
Raising the antiholomorphic indices in the tensor $F^\n_{KL}$ by the tensor $g^{lk}$ we see that the resulting tensor is $(-1)^{\n(1,1,-1)}\left(\hat \Gamma_\n\right)_K^I$, where $\left(\hat \Gamma_\n\right)_K^I$ is the matrix of the partition operator $\hat \Gamma_\n$ of the graph $\Gamma_\n$,
\[
        (-1)^{\n(1,1,-1)}\left(\hat \Gamma_\n\right)_{k_1 \ldots k_\qq}^{i_1 \ldots i_\pp} =   F^\n_{k_1 \ldots k_\qq l_1 \ldots l_\pp} g^{l_1 i_1}\ldots g^{l_\pp i_\pp},
\]
whence
\begin{equation}\label{E:gfg}
              (-1)^{\n(1,1,-1)}  \pp! \nu^\pp \left(\hat \Gamma_\n\right)_K^I =  F^\n_{KL} G^{LI}.
\end{equation}
We see from (\ref{E:eki}) and (\ref{E:mult}) that the contribution to the matrix $E_K^I$ of the operator $E$ corresponding to the multiplicity function $\n$ is
\[
     \prod_{(p,q,r) \in \mathbf{T}'} \frac{1}{\n(p,q,r)!} \left(\frac{(-1)\nu^r}{p!q!}\right)^{\n(p,q,r)}F^\n_{KL} G^{LI}.
\]
It follows from (\ref{E:bij}), (\ref{E:aut}), and (\ref{E:gfg}) that this contribution is equal to the contribution of  the graph $\Gamma_\n$ to the matrix of the operator on the right-hand side of (\ref{E:thm}), which implies the statement of the theorem.
\end{proof}
In the rest of the paper we will give a direct proof of Theorem \ref{T:operc}. Namely, we will prove that the operator on the right-hand side of (\ref{E:operc}) is inverse to the operator $E$. In order to work with  partition operators of graphs we use the formalism of Feynman and pre-Feynman diagrams developed in \cite{A}.

\section{Feynman diagrams}

In this section we consider Feynman and pre-Feynman diagrams related to weighted directed acyclic multi-graphs with one source and one sink, i.e., graphs from $\A$. A graph $\Gamma$ from $\A$ has nonnegative operator weight $\hat W(\Gamma)$ and each internal vertex has a weight from the set
\[
     W = \{-1,0,2,\ldots \}.
\] 
 We think of a directed edge as of consisting of two half-edges, a tail and a head. A vertex of a graph can be identified with its star (or a block), the set of heads of all incoming edges and tails of all outgoing edges. Given a graph $\Gamma$ over $\A$, we fix enumerations of the edges outgoing from the source and of the edges incoming to the sink of $\Gamma$. We call the set of half edges of $\Gamma$ with two partitions, into blocks and into edges, together with the enumerations of the source and sink blocks, the Feynman diagram of the graph $\Gamma$ and denote it $\F_\Gamma$.  The set of half edges of $\Gamma$ partitioned only into blocks is called the pre-Feynman diagram of $\Gamma$ and denoted $\E_\Gamma$. First we will define pre-Feynman diagrams as standalone objects, not attached to specific graphs.

A {\it pre-Feynman diagram} is a collection of data
\[
     \E = (E, E_h, E_t, V, V_{ext}, V_{int}, v, w),
\]
where $E$ is a finite set of half-edges partitioned into subsets $E_h$ of heads and $E_t$ of tails, $E = E_h\sqcup E_t$, with the same number of elements, $|E_h| = |E_t|$, $V$ is the set of vertices partitioned into a two-element set of external vertices $V_{ext}
= \{\i, \o\}$ (source $\i$ and sink $\o$) and a possibly empty set $V_{int}$ of internal verices, $v: E \to V$ is a surjective vertex mapping such that $v^{-1} (\i) \subset E_t, v^{-1}(\o)\subset E_h$, and $w : V_{int} \to W$ is a weight mapping. We also assume that for each internal vertex $a \in V_{int}$ the set $v^{-1}(a)$ contains at least one head and at least one tail and for each internal vertex of weight $-1$ the set $v^{-1}(a)$ has at least three elements. The set $v^{-1}(a)$ of half-edges is called the block of the vertex $a$. We fix enumerations on the $\i$-block $v^{-1}(\i)$ and on the $\o$-block $v^{-1}(\o)$, i.e., we identify the $\i$-block with the set $[\qq] = \{1,2,\ldots,\qq\}$ and the $\o$-block with $[\pp]= \{1,2,\ldots,\pp\}$, where $\pp = |v^{-1}(\o)|$ and $\qq =| v^{-1}(\i)|$.

We consider the groupoid category (i.e., where the morphisms are iso\-morphisms) {\it PreFey} of pre-Feynman diagrams where a morphism $f : \E \to \E'$ is given by bijections $f_E: E \to E'$
and $f_V:V \to V'$ which commute with the vertex mappings, $v' f_E = f_V v$, and are such that $f_V$ respects the weights of internal vertices, and $f_E$ respects the partition into heads and tails.
The mapping $f_E$ can change the enumerations of the $\i$- and $\o$-blocks.

The type of an internal vertex $a \in V_{int}$ of a pre-Feynman diagram is a triple $(p,q,r) \in \mathbf{T}$ such that $p = |v^{-1}(a) \cap E_h|$ is the number of heads and $q = |v^{-1}(a) \cap E_t|$ is the
number of tails in the block $v^{-1}(a)$ of $a$, and $r = w(a)$ is the weight of $a$. Morphisms respect the type of the vertices.

Given a pre-Feynman diagram $\E$, the multiplicity function $\m(p,q,r)$ returns the number of internal vertices of type $(p,q,r)$. The type of a pre-Feynman diagram $\E$ is a triple $(\pp,\qq, \m)$, where $\qq = |v^{-1}(\i)|, \pp = |v^{-1}(\o)|$, and $\m$ is the
multiplicity function. It determines $\E$ up to isomorphism. The group $G(\E)$ of automorphisms of a pre-Feynman diagram $\E$ of type $(\pp,\qq, \m)$ is of order 
\begin{equation}\label{E:grord}
 \gamma(\pp,\qq,\m) = \pp!\, \qq! \prod_{(p,q,r)\in \mathbf{T}} \m(p,q,r)! (p! q!)^{\m(p,q,r)}.
\end{equation}
It permutes heads and tails separately in each block and permutes blocks and vertices of the same type.

Given a pre-Feynman diagram $\E = (E, E_h, E_t, V, V_{ext},V_{int}, v, w)$, a Feynman diagram $\F$ over $\E$ is given by a
bijection $\sigma : E_t \to E_h$. A Feynman diagram $\F$ over $\E$ defines a directed graph $\Gamma_\F$ with the set of vertices $V$ and with
the edges given by the pairs $(x, \sigma(x)), x \in E_t$. We consider only the Feynman diagrams over $\E$ such that the
corresponding graphs are from $\A$. The type of a Feynman diagram $\F$ over a pre-Feynman diagram $\E$ is defined
to be the same as the type of $\E$ and of the  graph $\Gamma_\F$. Two Feynman diagrams $\F_1$ and $\F_2$ over pre-Feynman diagrams $\E_1$ and $\E_2$, respectively, are isomorphic if there is an underlying isomorphism of $\E_1$ and $\E_2$ which respects the partitions of the sets of half edges $E_1$ and $E_2$ into edges.

A graph $\Gamma$ from $\A$ determines  a pre-Feynman diagram $\E_\Gamma$ and a Feynman diagram $\F_\Gamma$ over $\E_\Gamma$ if we specify enumerations of the outgoing edges of the source and of the incoming edges of the sink of $\Gamma$.

The group $G(\E)$ acts upon Feynman diagrams over $\E$. The $G(\E)$-orbit $G(\E)\F$ of a Feynman diagram $\F$ over $\E$ consists of all Feynman diagrams over $\E$ that are isomorphic to $\F$. It will be denoted $[\F]$.
The stabilizer of a Feynman diagram $\F$ in the group $G(\E)$ is called the automorphism group of $\F$ and denoted $Aut (\F)$. It is naturally isomorphic to the group $Aut(\Gamma_\F)$.

\section{A composition formula}

In this section we give a detailed proof of a composition formula for the operators on the formal Fock space admitting a representation (\ref{E:ops}). The proof uses the formalism of pre-Feynman and Feynman diagrams developed in \cite{A}.

Given graphs $\Gamma_1$ and $\Gamma_2$ from $\A$ of types $(\pp_1,\qq_1,\m_1)$ and  $(\pp_2,\qq_2,\m_2)$, respectively, the product of the  partition operators $\hat\Gamma_1$ and $\hat\Gamma_2$ is nonzero only if $\pp_1 = \qq_2$. The product $\hat\Gamma_1\hat\Gamma_2$ is equal to a linear combination of partition operators of graphs from $\A$ whose Feynman diagrams are expressed in terms of an operation of concatenation $\#$ of Feynman diagrams of $\Gamma_1$ and $\Gamma_2$ introduced below.

We call two Feynman diagrams $\F_1$ and $\F_2$ of types $(\pp_1,\qq_1,\m_1)$ and  $(\pp_2,\qq_2,\m_2)$ composable if $\pp_1 = \qq_2$. Given two composable Feynman diagrams $\F_1$ and $\F_2$ of types $(\pp_1,\qq_1,\m_1)$ and  $(\pp_2,\qq_2,\m_2)$, respectively, we define a Feynman diagram $\F= \F_1 \# \F_2$ of type $(\pp_2,\qq_1,\m_1 + \m_2)$ as follows. We remove the $\o$-block of $\F_1$ and the $\i$-block of $\F_2$ and then connect a loose tail from $E_1$ with a loose head from $E_2$ if they were connected to a removed head and tail with the same number with respect to the enumerations of $v_1^{-1}(\o_1)$ and $v_2^{-1}(\i_2)$. 
A trivial but crucial observation is that if the graphs $\Gamma_{\F_1}$ and $\Gamma_{\F_2}$ are from $\A$, then so is the graph $\Gamma_{\F_1\# \F_2}$.

Let $\F_1$ and $\F_2$ be composable Feynman diagrams  of types $(\pp_1,\qq_1,\m_1)$ and  $(\pp_2,\qq_2,\m_2)$ over pre-Feynman diagrams $\E_1$ and $\E_2$ whose graphs $\Gamma_{\F_1}$ and $\Gamma_{\F_2}$ are from $\A$. Set $n = \pp_1 = \qq_2$ and consider the action of the symmetric group $S_n$ on the $\i$-block $v_2^{-1}(\i_2)$ of $\E_2$ which is identified with the set $[n]$ via an enumeration. The group $S_n$  acts on $\E_2$ by automorphisms and hence acts on the set of Feynman diagrams over $\E_2$. It can be easily verified that
\[
       \hat \Gamma_{\F_1}  \hat \Gamma_{\F_2} =  \frac{1}{n!}\sum_{\tau \in S_n} \hat \Gamma_{\F_1 \# \tau\F_2}.
\]
Denote by $\A(\pp,\qq,\m)$ the set of isomorphism classes of graphs from $\A$ of type $(\pp,\qq,\m)$. It follows that if $\Gamma_1$ and $\Gamma_2$ are graphs from $\A(n,\qq,\m_1)$ and $\A(\pp,n,\m_2)$, respectively, then the product $\hat\Gamma_1\hat\Gamma_2$ of their partition operators is a finite linear combination of partition operators of graphs from $\A(\pp,\qq,\m_1 + \m_2)$. 

Consider a series
\begin{equation}\label{E:format}
      \sum_{[\Gamma]\in\A} \frac{\nu^{\hat W(\Gamma)}a_\Gamma}{|Aut(\Gamma)|}  \hat \Gamma = \sum_{(\pp,\qq,\m)} \nu^{\hat W(\pp,\qq,\m)} \sum_{[\Gamma] \in \A(\pp,\qq,\m)}  \frac{a_\Gamma}{|Aut(\Gamma)|}  \hat \Gamma,
\end{equation}
where $a_\Gamma$ is a complex-valued function on $\A$. According to Lemma \ref{L:ops}, it determines a continuous operator on the formal Fock space $\V$, which will be denoted by $A$. 

Let  $\F$ be a Feynman diagram of a graph $\Gamma$ of type $(\pp,\qq,\m)$ over  a pre-Feynman diagram $\E$ of type $(\pp,\qq,\m)$. The $G(\E)$-orbit of $\F$ is the set of all Feynman graphs over $\E$ isomorphic to $\F$. It is denoted $[\F]$. We have
\begin{eqnarray*}
     \hat\Gamma = \frac{1}{|G(\E)|}\sum_{g \in G(\E)} \hat \Gamma_{g\F} = \frac{|Aut(\F)|}{|G(\E)|}\sum_{[g] \in G(\E)/Aut(\F)} \hat \Gamma_{g\F}\\ =  \frac{|Aut(\F)|}{|G(\E)|}\sum_{\F' \in [\F]} \hat \Gamma_{\F'}.
\end{eqnarray*}
Using formula (\ref{E:grord}) we get
\[
    A = \sum_{(\pp,\qq,\m)} \frac{\nu^{\hat W(\pp,\qq,\m)}}{\gamma(\pp,\qq,\m)}\sum_{\F'} a_{\Gamma_{\F'}} \hat \Gamma_{\F'},
\]
where $\F'$ runs over all Feynman diagrams over a fixed pre-Feynman diagram $\E(\pp,\qq,\m)$ of type $(\pp,\qq,\m)$ such that $\Gamma_{\F'}$ is from $\A(\pp,\qq,\m)$.

Now we want to study products of operators written in the format of (\ref{E:format}). Consider operators
\[
      A = \sum_{[\Gamma]\in\A} \frac{\nu^{\hat W(\Gamma)}a_\Gamma}{|Aut(\Gamma)|}  \hat \Gamma  \mbox{ and }  B = \sum_{[\Gamma]\in\A} \frac{\nu^{\hat W(\Gamma)}b_\Gamma}{|Aut(\Gamma)|}  \hat \Gamma ,
\]
where $a_\Gamma$ and  $b_\Gamma$ are complex-valued functions on $\A$. The product of these operators can be written as follows:
\begin{eqnarray*}
  AB = \sum_{[\Gamma_1],[\Gamma_2]\in\A}\frac{\nu^{\hat W(\Gamma_1) + \hat W(\Gamma_2)}a_{\Gamma_1} b_{\Gamma_2}}{|Aut(\Gamma_1)||Aut(\Gamma_2)|}\hat \Gamma_1 \hat \Gamma_2 = \\
\sum_{(\pp,\qq,\m)}  \sum_{\m' \leq \m}\sum_{[\Gamma_1]\in \A(n,\qq,\m')}\sum_{[\Gamma_2]\in \A(\pp,n,\m-\m')}\nu^{\hat W(\pp,\qq,\m)}Y(\Gamma_1,\Gamma_2),
\end{eqnarray*}
where the inequality $\m' \leq \m$ is pointwise, i.e., $\m'(p,q,r) \leq \m(p,q,r)$ for all internal vertex types $(p,q,r)$,
\begin{equation}\label{E:yy}
  Y(\Gamma_1,\Gamma_2) = \frac{a_{\Gamma_1} b_{\Gamma_2}}{|Aut(\Gamma_1)||Aut(\Gamma_2)|}\hat \Gamma_1 \hat \Gamma_2,
\end{equation}
 and $n$ is uniquely determined by the formula
\begin{eqnarray}\label{E:nnn}
   n = \qq + \sum_{(p,q,r) \in \mathbf{T}} (q-p)\m'(p,q,r) =\\
\pp + \sum_{(p,q,r) \in \mathbf{T}}  (p-q) (\m(p,q,r) - \m'(p,q,r)).\nonumber
\end{eqnarray} 
We have used that for $\Gamma_1$ from $\A(n,\qq, \m')$ and $\Gamma_2$ from $\A(\pp, n, \m - \m')$,
\[
     \hat W(\Gamma_1)+\hat W(\Gamma_2) = \hat W(\pp,\qq,\m).
\]
 We want to simplify formula (\ref{E:yy}), where we assume that $\Gamma_1$ is of type $(n,\qq,\m')$ and $\Gamma_2$ is of type $(\pp,n,\m-\m')$ for a fixed $\m'$ satisfying $\m' \leq \m$. For $i =1,2$,  let $\F_i$  be a Feynman diagram of the graph $\Gamma_i$ over a pre-Feynman diagram $\E_i$. We will use formula (\ref{E:grord}) and the notation
\[
   \binom {\m}{\m'} = \prod_{(p,q,r) \in \mathbf{T}} \binom{\m(p,q,r)}{\m'(p,q,r)}.
\]
We have
\begin{eqnarray*}
    Y(\Gamma_1,\Gamma_2) =\frac{1}{|G(\E_1)||G(\E_2)|}\sum_{g_1 \in G(\E_1)} \sum_{g_2 \in G(\E_2)}\frac{a_{\Gamma_{\F_1}} b_{\Gamma_{\F_2}}}{|Aut(\F_1)||Aut(\F_2)|}\hat\Gamma_{g_1\F_1}\hat\Gamma_{g_2\F_2}\\
=\frac{1}{|G(\E_1)||G(\E_2)|}\sum_{g_1 \in G(\E_1)} \sum_{g_2 \in G(\E_2)}\frac{1}{n!}\sum_{\tau \in S_n}\frac{a_{\Gamma_{\F_1}} b_{\Gamma_{\F_2}}}{|Aut(\F_1)||Aut(\F_2)|}\hat\Gamma_{g_1\F_1 \# \tau g_2\F_2}\\
=\frac{1}{|G(\E_1)||G(\E_2)|}\sum_{g_1 \in G(\E_1)} \sum_{g_2 \in G(\E_2)}\frac{a_{\Gamma_{\F_1}} b_{\Gamma_{\F_2}}}{|Aut(\F_1)||Aut(\F_2)|}\hat\Gamma_{g_1\F_1 \# g_2\F_2}\\
\sum_{[g_1] \in G(\E_1)/Aut(\F_1)} \sum_{[g_2] \in G(\E_2)/Aut(\F_2)}\frac{a_{\Gamma_{\F_1}} b_{\Gamma_{\F_2}}}{|G(\E_1)||G(\E_2)|}\hat\Gamma_{g_1\F_1 \# g_2\F_2}=\\
\sum_{\F' \in [\F_1]} \sum_{\F'' \in [\F_2]}\frac{a_{\Gamma_{\F_1}} b_{\Gamma_{\F_2}}}{\gamma(n,\qq,\m')\gamma(\pp,n,\m-\m')}\hat\Gamma_{\F' \# \F''}=\\
\frac{1}{\gamma(\pp,\qq,\m)}\sum_{\F' \in [\F_1]} \sum_{\F'' \in [\F_2]} \binom {\m}{\m'}\frac{a_{\Gamma_{\F_1}} b_{\Gamma_{\F_2}}}{(n!)^2} \hat\Gamma_{\F' \# \F''}.
\end{eqnarray*}
Now we can write the product $AB$ as follows:
\begin{equation}\label{E:prodab1}
  AB = \sum_{(\pp,\qq,\m)} \frac{\nu^{\hat W(\pp,\qq,\m)}}{\gamma(\pp,\qq,\m)}  \sum_{\m' \leq \m}\sum_{\F',\F''}\binom {\m}{\m'}\frac{a_{\Gamma_{\F'}} b_{\Gamma_{\F''}}}{(n!)^2} \hat\Gamma_{\F' \# \F''},
\end{equation}
where $\F'$ runs over all Feynman diagrams over a fixed pre-Feynman diagram $\E_1(n,\qq,\m')$ of type $(n,\qq,\m')$ such that $\Gamma_{\F'}$ is from $\A(n,\qq,\m')$ and $\F''$ runs over all Feynman diagrams over a fixed pre-Feynman diagram $\E_2(\pp,n,\m-\m')$ of type $(\pp,n,\m-\m')$ such that $\Gamma_{\F''}$ is from $\A(\pp,n,\m-\m')$.

Formula (\ref{E:prodab1}) can be specified further. For each type $(\pp,\qq,\m)$ fix a pre-Feynman diagram $\E(\pp,\qq,\m)$ of that type. Let $\pi$ be a partition of the set $E_{int}$ of internal vertices of $\E(\pp,\qq,\m)$ into an ordered pair of subsets $E'_{int}$ and $E''_{int}$ with multiplicity functions $\m'_\pi$ and $\m-\m'_\pi$, respectively. Construct a pre-Feynman diagram $\E'_\pi$ of type $(n_\pi,\qq,\m'_\pi)$ from the internal blocks of $\E(\pp,\qq,\m)$ corresponding to $E'_{int}$,  $\i$-block $[\qq]$, and $\o$-block $[n_\pi]$. Then construct a pre-Feynman diagram $\E''_\pi$ of type $(\pp,n_\pi,\m-\m'_\pi)$ from the internal blocks of $\E(\pp,\qq,\m)$ corresponding to $E''_{int}$,  $\i$-block $[n_\pi]$, and $\o$-block $[\pp]$. Formula (\ref{E:prodab1}) can be rewritten as follows:
\begin{equation}\label{E:prodab2}
   AB = \sum_{(\pp,\qq,\m)} \frac{\nu^{\hat W(\pp,\qq,\m)}}{\gamma(\pp,\qq,\m)}\sum_\F \left(\sum_\pi \sum_{\F',\F''} \frac{a_{\Gamma_{\F'}} b_{\Gamma_{\F''}}}{(n_\pi!)^2}\right) \hat \Gamma_{\F},
\end{equation}
where for each type $(\pp,\qq,\m)$ we fix $\E(\pp,\qq,\m)$, $\F$ runs over the Feynman diagrams over $\E(\pp,\qq,\m)$ such that $\Gamma_\F$ is from $\A$, $\pi$ is a partition of the internal blocks of $\E(\pp,\qq,\m)$, and $\F'$ and $\F''$ run over all Feynman diagrams over $\E'_\pi$ and $\E''_\pi$, respectively, such that $\Gamma_{\F'}$ and $\Gamma_{\F''}$ are from $\A$ and $\F' \# \F'' = \F$.

Assume that $\pi$ is a partition  of the internal blocks of $\E(\pp,\qq,\m)$ such that there is at least one pair $(\F',\F'')$ of Feynman diagrams corresponding to $\pi$ such that $\F' \# \F'' = \F$.  The Feynman diagrams $\F'$ and $\F''$ are almost completely determined by the partition $\pi$. Namely, the only edges in $\F'$ that are not inherited from $\F$ are those that are incident to the $\o$-block of $\F'$, and the edges of $\F''$ not inherited from $\F$ are those incident to the $\i$-block of $\F''$. Consider an edge in $\F$ that connects a block that belongs to $\E'_\pi$ (either the $\i$-block or an internal block) with a block that belongs to $\E''_\pi$ (either an internal block or the $\o$-block). In $\F'$,  its tail is connected with a head in the $\o$-block $[n_\pi]$. In $\F''$ its head is connected with the tail in the $\i$-block $[n_\pi]$ with the same number. Thus, for each partition $\pi$, there are either zero or $n_\pi!$ pairs of Feynman diagrams $(\F',\F'')$ such that $\F' \# \F'' = \F$, corresponding to different enumerations of the loose ends. Moreover, the isomorphism classes of $\F'$ and $\F''$ do not depend on these enumerations. Therefore, (\ref{E:prodab2}) can be simplified even more: 
\begin{equation}\label{E:prodab3}
     AB = \sum_{(\pp,\qq,\m)} \frac{\nu^{\hat W(\pp,\qq,\m)}}{\gamma(\pp,\qq,\m)}\sum_\F \left(\sum_\pi \frac{a_{\Gamma_{\F'}} b_{\Gamma_{\F''}}}{n_\pi!}\right) \hat \Gamma_{\F},
\end{equation}
where $\F$ runs over the Feynman diagrams over $\E(\pp,\qq,\m)$ such that $\Gamma_\F$ is from $\A$ and $(\F',\F'')$ is {\it any} pair of Feynman diagrams corresponding to the partition $\pi$ such that $\F' \# \F'' = \F$. 

Eventually we group together the summands in the middle sum in (\ref{E:prodab3}) corresponding to equivalent Feynman diagrams. Given a graph $\Gamma$ from $\A$, a partition $\pi$ of the set $E_{int}$ of internal vertices of $\Gamma$  into subsets $E'_\pi$ and $E''_\pi$ will be called admissible if any edge connecting a vertex from $\i \cup  E'_\pi$ with a vertex from $E''_\pi \cup \o$ has a tail in $\i \cup  E'_\pi$ and a head in  $E''_\pi \cup \o$. We denote the number of such edges also by $n_\pi$. Formula (\ref{E:prodab3}) can be finally rewritten in terms of graphs,
\begin{equation}\label{E:prodab4}
       AB = \sum_{[\Gamma] \in \A} \frac{\nu^{\hat W(\Gamma)}}{|Aut(\Gamma)|}\left(\sum_{\pi} \frac{a_{\Gamma'_\pi}b_{\Gamma''_\pi}}{n_\pi !}\right) \hat \Gamma,
\end{equation}
where $\pi$ runs over the admissible partitions of  $E_{int}$,  the graph $\Gamma'_\pi$ is obtained from $\Gamma$ by deleting all vertices from  $E''_\pi$ and all edges connecting vertices within $E''_\pi \cup \o$, and connecting the $n_\pi$ loose edges with the sink $\o$, and graph $\Gamma''_\pi$ is obtained from $\Gamma$ by deleting all vertices from  $E'_\pi$ and all edges connecting vertices within $\i \cup E'_\pi$, and connecting the $n_\pi$ loose edges with the source $\i$. Observe that both $\Gamma'_\pi$ and  $\Gamma''_\pi$ are from $\A$.

The composition formula  (\ref{E:prodab4}) is in the spirit of the product of combinatorial species in the sense of Joyal \cite{AJ}.

Denote by $\hat \A$ the set of continuous operators on the formal Fock space $\V$ admitting a representation (\ref{E:ops}).  According to Lemma \ref{L:ops}, all such linear combinations are continuous operators on $\V$. 

\begin{theorem}\label{T:ahat}
 The set $\hat\A$ is an algebra of continuous operators on the formal Fock space $\V$ with the composition formula  (\ref{E:prodab4}).
\end{theorem}
\begin{proof}
The theorem is an immediate consequence of Lemma \ref{L:ops} and formula (\ref{E:prodab4}).
\end{proof}

The algebra $\hat\A$ contains the operator $E$ and, as follows from \cite{G} and Theorem \ref{T:operc}, the inverse operator $C = E^{-1}$  given by formula (\ref{E:operc}) also lies in $\hat\A$. In the next section we will give a direct proof of this fact.

\section{An alternative proof of Gammelgaard's formula}

We want to show directly that the product of the right hand sides of (\ref{E:thm}) and  (\ref{E:operc}) ,
\begin{equation}\label{E:prod}
    \left(\sum_{[\Gamma] \in \U} \frac{(-1)^{R(\Gamma)}\nu^{\hat W(\Gamma)}\pp(\Gamma)!}{|Aut(\Gamma)|} \hat \Gamma\right)\left(\sum_{[\Gamma] \in \A} \frac{\nu^{\hat W(\Gamma)}\pp(\Gamma)!}{|Aut(\Gamma)|}\, \hat \Gamma\right),
\end{equation}
is the identity operator on $\V$. Assuming that $a_\Gamma$ is supported on $\U$,  $b_\Gamma$ is supported on $\A$, 
\[
    a_\Gamma = (-1)^{R(\Gamma)}\pp(\Gamma)!, \mbox{ and } b_\Gamma = \pp(\Gamma)!,
\]
we get that formula (\ref{E:prodab4}) gives the product (\ref{E:prod}).
Fix a type $(\pp,\qq,\m)$ and a graph $\Gamma$ from $\A(\pp,\qq,\m)$. Then the sum in the parentheses in formula (\ref{E:prodab4}) corresponding to $\Gamma$ is
\begin{equation}\label{E:suminpar}
   \sum_\pi \frac{a_{\Gamma'_\pi} b_{\Gamma''_\pi}}{n_\pi!} =  \pp!  \sum_\pi  (-1)^{R(\Gamma'_\pi)},
\end{equation}
where $\pi$ in the sum on the right-hand side of (\ref{E:suminpar}) runs over the admissible partitions of the set $E_{int}$ of internal vertices of $\Gamma$ such that $\Gamma'_\pi$ is from $\U$.

Denote by $S$ the set of all internal vertices of $\Gamma$ which have no incoming edges that are outgoing from other internal vertices. Let $\d$ be the multiplicity function of $S$. Given an admissible partition $\pi$ of the set $E_{int}$ of internal vertices of $\Gamma$ such that $\Gamma'_\pi$ is from $\U$, then  $E'_\pi \subset S$.  Vice versa, if $E'$ is any subset of $S$, then the partition $\pi$ of $E_{int}$ into the subsets $E'_\pi = E'$  and $E''_\pi =E_{int} \backslash E'$ is admissible and the graph $\Gamma'_\pi$ is from $\U$. If $\m \neq 0$, i.e., $\m(p,q,r) \neq 0$ for at least one type $(p,q,r)$, then the set $S$ is nonempty, $\d \neq 0$, and we have from the abovementioned considerations that
\begin{eqnarray*}
    \sum_{\pi} (-1)^{R(\Gamma'_\pi)} = \sum_{\m' \leq \d} (-1)^{\sum_{(p,q,r)\in \mathbf{T}} \m'(p,q,r)} \prod_{(p,q,r)\in \mathbf{T}}\binom{\d(p,q,r)}{\m'(p,q,r)}\\
= \prod_{(p,q,r)\in \mathbf{T}} \sum_{i=0}^{\d(p,q,r)} (-1)^i \binom{\d(p,q,r)}{i}=0.
\end{eqnarray*}
Thus the summands in   (\ref{E:prodab4}) corresponding to the graphs from  $\A(\pp,\qq,\m)$ with $\m \neq 0$ do not contribute to the product (\ref{E:prod}).

Consider a graph of type $(\pp,\qq,\m)$ with $\m=0$, i.e.,  with no internal vertices. Then $\pp=\qq$ and we denote their common value by $n$. We used the notation $\Lambda_n$ for such a graph.  We have $R(\Lambda_n) = 0, \hat W(\Lambda_n) = 0,$ and $|Aut(\Lambda_n)| = n!$. The set $\A(n,n,0)$ consists of the single isomorphism class of the graph $\Lambda_n$.  It follows from (\ref{E:suminpar}) that the contribution to  (\ref{E:prodab4}) corresponding to the type $(n,n,0)$ is $ \hat \Lambda_n$. According to (\ref{E:ident}), the total contribution to (\ref{E:prod}) corresponding to the types $(n,n,0), n \geq 0$,  is the identity operator. Thus we have shown that the operator on the formal Fock space $\V$ given by the right-hand side of (\ref{E:operc}) is right-inverse to the invertible operator $E$ and therefore is equal to the operator $C$.


\begin{thebibliography}{99}
\bibitem{A} Abdesselam, A.: Feynman diagrams in algebraic combinatorics. {\it S\' em. Lothar. Combin.} {\bf49} (2002/04), Art. B49c, 45 pp. (electronic)
\bibitem{AL}Alekseev, A., Lachowska, A.: Invariant *-products on coadjoint orbits and the Shapovalov pairing.
{\it Comment. Math. Helv.} {\bf 80} (2005),  795--810.
\bibitem{BFFLS} Bayen, F., Flato, M., Fronsdal, C.,
Lichnerowicz, A., and Sternheimer, D.: Deformation theory and
quantization. I. Deformations of symplectic structures. {\it
Ann. Physics} {\bf 111} (1978), no. 1, 61 -- 110.
\bibitem{Ber} Berezin, F.A.: Quantization in complex symmetric spaces. {\it Math. USSR-Izv.} {\bf 39} (1975), 363--402.
\bibitem{BBEW} Bordemann, M., Brischle, M.,  Emmrich, C., and
Waldmann, S.: Phase space reduction for star-products: An explicit
construction for $\mathbf{P}^n$. {\it Lett. Math. Phys.} {\bf 36}, (4)(1996), 357--371.
\bibitem{BW} Bordemann, M., Waldmann, S.: A Fedosov star product
of the Wick type for K\"ahler manifolds. {\it Lett. Math. Phys.}
{\bf 41} (3) (1997), 243 -- 253.
\bibitem{CGR} Cahen, M., Gutt S., and Rawnsley, J.: Quantization of K\"ahler manifolds II. {\it Trans. Amer. Math. Soc.} {\bf 337} (1993), 73--98.
\bibitem{F} Fedosov, B.:  A simple geometrical construction of
deformation quantization.
{\it J. Differential Geom.} {\bf 40}  (1994),  no. 2, 213--238.
\bibitem{G} Gammelgaard, N. L.: A Universal Formula for
Deformation Quantization on K{\"a}hler Manifolds,
arXiv:1005.2094v2.
\bibitem{GR} Gutt, S. and Rawnsley, J.: Natural star products on symplectic manifolds and quantum moment maps. {\it Lett. Math. Phys.} {\bf 66}(2003), 123 --139.
\bibitem{AJ} Joyal, A.: Une th\' eorie combinatoire des s\'eries formelles. {\it Adv. Math} {\bf 42}(1981), 1 -- 82.
\bibitem{CMP1} Karabegov, A.: Deformation quantizations with
separation of variables on a K\"ahler manifold.
{\it Commun. Math. Phys.} {\bf 180}  (1996),  no. 3, 745--755.
\bibitem{CMP3} Karabegov, A.: Formal symplectic groupoid of a deformation quantization. {\it Commun. Math. Phys.} {\bf 258} (2005), 223--256. 
\bibitem{KSch} Karabegov, A., Schlichenmaier, M.: Identification of Berezin-Toeplitz deformation quantization. {\it J. reine angew. Math.} {\bf 540} (2001), 49-76.
\bibitem{K} Kontsevich, M.: Deformation quantization of Poisson
manifolds, I.
 {\it Lett. Math. Phys.} {\bf 66} (2003), 157 -- 216.
\bibitem{Mor} Moreno, C.: $*$-products on some K\"ahler manifolds. {\it Lett. Math. Phys.} {\bf 11} (1986), 361--372.
\bibitem{N} Neumaier, N.: Universality of Fedosov's construction for star products of Wick type on pseudo-K\"ahler manifolds. {\it Rep. Math. Phys.} {\bf 52}(2003), 43--80.
\bibitem{RT} Reshetikhin, N., Takhtajan, L.: Deformation
quantization of K\"ahler manifolds.
L. D. Faddeev's Seminar on Mathematical Physics, Amer. Math.
Soc. Transl. Ser. 2, {\bf 201}, Amer. Math. Soc., Providence,
RI, (2000), 257--276.
\bibitem{Sch} Schlichenmaier, M.: Berezin-Toeplitz quantization of compact K\"ahler manifolds. In: Quantization, Coherent States and Poisson Structures, Proc. XIV'th Workshop on Geometric Methods in Physics (Bialowieza, Poland, 9-15 July 1995), A. Strasburger, S. T. Ali, J.-P. Antoine, J.-P. Gazeau, and A. Odzijewicz, eds., Polish Scientific Publisher PWN (1998), 101 -- 115.
\bibitem{Xu1} Xu, H.: An explicit formula for the Berezin star product, arXiv:1103.4175 (to appear in {\it Lett. Math. Phys.}).
\bibitem{Xu2} Xu, H.: On a formula of Gammelgaard for Berezin-Toeplitz quantization, arXiv:1204.2259.
\end{thebibliography}
\end{document}